\documentclass[12pt,english]{article} 
\usepackage{times} 
\usepackage[]{graphicx} 

\usepackage[T1]{fontenc}
\usepackage[latin9]{inputenc}
\usepackage{amsthm}
\usepackage{amsmath}
\usepackage{amssymb}

\theoremstyle{plain}
\theoremstyle{plain}
\newtheorem{thm}{Theorem}
  \theoremstyle{plain}
  \newtheorem{lem}[thm]{Lemma}
  \theoremstyle{plain}
  \newtheorem{cor}[thm]{Corollary}

\usepackage{amsthm}

\usepackage{amsfonts}
\usepackage{latexsym}

\title{On a conjecture of  Parker about Dessins d'Enfants}
\author{Corneliu Hoffman Department of Mathematics and Statistics 
\\ Bowling Green State
University \\ Bowling Green, OH 43403-1874}

\usepackage{babel}

\begin{document}
\title{On a conjecture of  Parker about Dessins d'Enfants} \author{Corneliu Hoffman \\ School of Mathematics  \\  University of Birmingham\\Edgbaston B15 2NN} \maketitle
\begin{abstract}
This note is concerned with the disproof of the most general case
of Parker's conjecture as stated in \cite{R04,LocSch06}. The conjecture
relates a certain group theoretic objects to the field of moduli of
a Dessin.
\end{abstract}
In his famous memoir \textquotedbl{}Esquisse d'un programme\textquotedbl{}
( a translation of which can be find in \cite{G97}), Alexandre Grothendieck
proposed the study of the absolute Galois group $\mbox{Gal}(\bar{\mathbb{Q}}/\mathbb{Q})$
via its faithful action on a collection of combinatorial objects that
he called \textit{dessins d'enfants} or children drawings.

The idea is based on a theorem of Belyi (see\cite{B79} or \cite{S94})
which states that an a complex algebraic curve $X$ can be defined
over $\bar{\mathbb{Q}}$ if and only if it admits a map $\beta:X\to\mathbb{P}^{1}$,
defined over $\bar{\mathbb{Q}}$ and ramified only over the points
$0,1,\infty$. Such a map is called a Belyi map and the pair $(X,\beta)$
is called a Belyi pair. Of course the group $\mbox{Gal}(\bar{\mathbb{Q}}/\mathbb{Q})$
has a natural action of Belyi pairs.

Given a Belyi pair one can construct a bipartite graph on the surface
$X$ lifting the segment $0,1$ from $\mathbb{P}^{1}$. This is the
dessin that corresponds to the Belyi pair. The two notions are completely
equivalent more precisely there is a natural (but nontrivial) way
to reconstruct the pair from the dessin. Therefore the group $\mbox{Gal}(\bar{\mathbb{Q}}/\mathbb{Q})$
acts on the collection or dessins, that is those bipartite graphs
on surfaces with the property that the complement of the dessin is
a union of simply connected cells (corresponding to the points above
$\infty$).

Finally given a Belyi pair/dessin, one can construct the action of
the fundamental group of $\mathbb{P}^{1}-\{0,1,\infty\}$ on a generic
fiber. This associates a permutation group called the monodromy group
to the dessin. More precisely the fundamental group $\pi_{1}(\mathbb{P}^{1}-\{0,1,\infty\},p)=\langle a,b,c|abc=1\rangle$
where $a,b,c$ are standard generators of the fundamental group (loops
around $0,1$ and $\infty$). The corresponding pair of permutations
given by the images of $a$ and $b$ define the dessin up to isomorphism.
Therefore you get an action of $\mbox{Gal}(\bar{\mathbb{Q}}/\mathbb{Q})$
on this collection of pairs of permutations.

We refer to the survey \cite{S94} for details on the various constructions.

One of the central questions in the area is finding \textquotedbl{}good\textquotedbl{}
invariants of a dessin. That is, finding invariants that will differentiate
between dessins that are not Galois conjugate.

One interesting invariant is the \textit{field of moduli}. Given a
dessin $D$ consider the stabiliser of $D$ in $\mbox{Gal}(\bar{\mathbb{Q}}/\mathbb{Q})$,
that is the group $\Gamma_{D}:=\{g\in\mbox{Gal}(\bar{\mathbb{Q}}/\mathbb{Q})|D^{g}=D\}$.
Note that the dessin is not just the curve but rather the curve $X$
but the Belyi cover and the monodromy group. The group $\Gamma_{D}$
is then the stabiliser of the triple $(X,\beta,G)$ where The field
of moduli of $D$ is the field $\mbox{Fix}(\Gamma_{D})$.

A field $K$ is called a field of definition for $D$ if there exists
a Belyi pair for $D$ so that $X$, $\beta$ and the monodromy group
are defined over $K$. Equivalently the field of moduli of a dessin
is the intersection of all its fields of definitions. The field of
moduli does not need to be a field of definition. It is quite hard
in practice to compute the field of moduli of various dessins and
Parker's conjecture proposes an alternative method.

As before a dessin $D$ can be completely described by a pair of permutations.
If the dessin $D$ is given by two permutations $a,b$ and $G$ is
the monodromy group then consider the element $x=\sum_{g\in G}(g^{-1}ag,g^{-1}bg)\in\mathbb{Q\mathbb{}}[G\times G]$.
R. Parker conjectured in 1984 that the field of moduli of $D$ is
generated over $\mathbb{Q}$ by the eigenvalues of $x$ in its action
on $\mathbb{Q}[G\times G]$. This was listed as one of the remarkable
open problems in the field in the field (see\cite{S94,R04,LocSch06}). 

In an unpublished note, L Schneps proved the conjecture for the case
of genus zero dessins. In each of those cases however, the field of
moduli was abelian. The aim of this note is to show that this needs
to be the case, see Corollary \ref{cor:Parker's-conjecture}and the
following weaker version of the conjecture.

\paragraph{The group theory}

Consider a group $G$ and two elements $a,b\in G$. Moreover consider
the element $x=\sum_{g\in G}(g^{-1}ag,g^{-1}bg)\in\mathbb{Q}[G\times G]$
in its left action on $\mathbb{Q}[G\times G]$. We will extend the
field of constants to $\mathbb{C}$ and work inside $\mathbb{C}[G]$
respectively $\mathbb{C}[G\times G]$. The following is rather obvious. 
\begin{lem}
The element $x$ commutes with the diagonal copy of $\mathbb{C}[G]$
in $\mathbb{C}[G\times G]$. Note that if we regard $\mathbb{C}[G\times G]$
as a $\mathbb{C}[G]$-module it is isomorphic to $\mathbb{C}[G]\otimes\mathbb{C}[G]$.
If $U,V$ are irreducible representations of $G$ then $U\otimes V$
is a submodule of \textup{$\mathbb{C}[G]\otimes\mathbb{C}[G]$}
In particular if we denote by $\lambda$ the trivial representation
of $G$ then both $\lambda\otimes V$ and $V\otimes\lambda$ are irreducible
$\mathbb{C}[G]$-submodules of $\mathbb{C}[G]\otimes\mathbb{C}[G]$
and they are isomorphic to $V$. \end{lem}
\begin{cor}
The element $x$ permutes the irreducible $\mathbb{C}[G]$-submodules
of $\mathbb{C}[G\times G]$. In particular the eigenspaces of $x$
are $\mathbb{C}[G]$-submodules of $\mathbb{C}[G\times G]$. \end{cor}
\begin{lem}
If $U,V$ are irreducible representations of $G$ then $U\otimes V$
is a submodule of $\mathbb{C}[G]\otimes\mathbb{C}[G]$ that is invariant
under $x$. If $W$ is an irreducible $\mathbb{C}[G]$-submodules
of $\mathbb{C}[G\times G]$ that is invariant under $x$ then $W$
is included in an eigenspace of $x$.\end{lem}
\begin{proof}
The first statement is an easy verification. To show the second assertion
one notes that $x\in\mbox{End}_{\mathbb{C}[G]}(V)\cong\mathbb{C}$
from Schur's lemma. \end{proof}
\begin{cor}
If $V$ is an irreducible representation of $G$ then $x$ leaves
both $\lambda\otimes V$ and $V\otimes\lambda$ invariant. Moreover
the two modules will in fact be eigenspaces for $x$ and the respective
eigenvalues will be rational multiples of the character values of
$b$ respectively $a$ on the module $V$. \end{cor}
\begin{proof}
Of course if $1\otimes v\in\lambda\otimes V$ then $x(1\otimes v)=\sum_{g\in G}1\otimes g^{-1}bg(v)\in\lambda\otimes V$.
Moreover $\lambda\otimes V$ is simple as a $\mathbb{C}[G]$ module
hence by the above lemma $x$ will act as a scalar on the simple module
$\lambda\otimes V$. This means that $\lambda\otimes V$ is an eigenspace
for $x$. If we examine the trace of action of $x$ on this module
we note that $\mbox{tr}(x)=|G|\mbox{tr}(b)=|G|\chi(b)$ where $\chi$
is the character of $V$. Moreover since $x$ is a scalar, $\mbox{tr}(x)=\chi(1)\rho$
where $\rho$ is the eigenvalue of $x$ on $\lambda\otimes V$. It
follows that $\rho=|G|\frac{\chi(b)}{\chi(1)}$. \end{proof}
\begin{cor}
The set of character values of the elements $a$ and $b$ are included
in the field obtained by adjoining the eigenvalues of $x$ to $\mathbb{Q}$.
\end{cor}
Consider $K$ a splitting field of $G$ (that is a field such that
any $K[G]$ irreducible module is absolutely irreducible).
\begin{lem}
The field generated by the eigenvalues of $x$ is contained in the
splitting field of $G$. 
\end{lem}
We will consider the natural inclusion $\mathbb{Q\mathbb{}}[G\times G]\subseteq K[G\times G]$
and view $x$ as an element of $K[G\times G]$. At the same time for
the purpose of finding the eigenvalues we also consider the natural
embedding $K[G\times G]\subseteq\mathbb{C}[G\times G]$.
\begin{proof}
Note that a $K[G]$-module $V$ is irreducible if and only if $\mbox{End}_{K[G]}(V)=K$
(see 29.13 of \cite{CR88}). Therefore if $V$ is an irreducible $K[G]$-module
that is invariant under $x$ then $x$ acts on it as a scalar in $K$.

Conversely consider $\bar{W}$ an eigenspace of $x$ (this is of course
a subspace of $\mathbb{C}[G\times G]$). There exists a subspace $\bar{V}\le\bar{W}$
that is an irreducible $\mathbb{C}[G]$-module. Since $K$ is the
splitting field of $G$ there exists an irreducible $K[G]$-module
$V$ such that $\bar{V}=V\otimes_{K}\mathbb{\mathbb{C}}$ (for example
see problem 28.9 in \cite{CR88}). Since $x$ acts as a scalar on
$\bar{V}$ it will fix $V$. Moreover the action of $x$ on $V$ depends
exclusively on the action of $G$ on $V$ so it will be $K[G]$-linear.
In particular $x$ acts on $V$ as an element of $\mbox{End}_{K[G]}(V)=K$
and so the eigenvalues of $x$ are in $K$.
\end{proof}
Combining the results we obtain the following: 
\begin{thm}
Let $L$ be the field generated over $\mathbb{Q}$ by all the eigenvalues
of $x$. Then $k\le L\le K$, where $k$ is the field generated by
the character values of $a$ and $b$ and $K$ is the field generated
by the $|G|$-roots of one. \end{thm}
\begin{cor}
\label{cor:Parker's-conjecture}Parker's conjecture can only hold
for dessins with abelian fields of moduli. \end{cor}


\begin{thebibliography}{8}
\bibitem{B79} G.~V. Bely{\u{\i}}. \newblock Galois extensions
of a maximal cyclotomic field. \newblock {\em Izv. Akad. Nauk SSSR
Ser. Mat.}, 43(2):267--276, 479, 1979.

\bibitem{CG94} J.-M. Couveignes and L.~Granboulan. \newblock Dessins
from a geometric point of view. \newblock In {\em The Grothendieck
theory of dessins d'enfants (Luminy, 1993)}, volume 200 of {\em
London Math. Soc. Lecture Note Ser.}, pages 79--113. Cambridge Univ.
Press, Cambridge, 1994.

\bibitem{CR88} C.~W. Curtis and I.~Reiner. \newblock {\em Representation
theory of finite groups and associative algebras}. \newblock Wiley
Classics Library. John Wiley \& Sons Inc., New York, 1988.

\bibitem{G97} A.~Grothendieck. \newblock Esquisse d'un programme.
\newblock In {\em Geometric Galois actions, 1}, volume 242 of
{\em London Math. Soc. Lecture Note Ser.}, pages 5--48. Cambridge
Univ. Press, Cambridge, 1997.

\bibitem{LocSch06}P.\textasciitilde{}Lochak and L.\textasciitilde{}Schneps.
\newblock Open problems in {G}rothendieck-{T}eichm\textbackslash{}\textquotedbl{}uller
theory. \newblock In {\em Problems on mapping class groups and
related topics}, volume\textasciitilde{}74 of {\em Proc. Sympos.
Pure Math.}, pages 165--186. Amer. Math. Soc., Providence, RI, 2006.

\bibitem{R04} O.~Röndigs. \newblock Theory of motives, homotopy
theory of varieties, and dessins d'enfants. \newblock workshop in
Palo Alto, 04 2004.

\bibitem{S94} L.~Schneps. \newblock Dessins d'enfants on the {R}iemann
sphere. \newblock In {\em The Grothendieck theory of dessins d'enfants
(Luminy, 1993)}, volume 200 of {\em London Math. Soc. Lecture Note
Ser.}, pages 47--77. Cambridge Univ. Press, Cambridge, 1994.

\bibitem{S07}L.~Schneps. \newblock Some notes on Parker\textquoteright{}s
conjecture, 2007 http://people.math.jussieu.fr/\textasciitilde{}leila/articles.html
\end{thebibliography}
\end{document}